\documentclass[12pt,a4paper]{amsart}
\usepackage{fullpage}
\usepackage{a4wide}
\usepackage[dvips]{graphicx}
\usepackage{amsmath}
\usepackage{amssymb,amsthm,amscd}
\usepackage{my}
\def\x{\bar{x}}
\def\y{\bar{y}}
\renewcommand\m{\bar{\mathfrak{m}}}
\renewcommand\subseteq{\subset}
\renewcommand\supseteq{\supset}
\def\mm{\mathfrak{m}}
\def\wt{{\rm wt}}

\begin{document}
\tolerance=500
\title{On solvability of the automorphism group \\
 of a %local
  finite-dimensional algebra}
\author{Alexander Perepechko}
\email{perepechko@mccme.ru}
\address{Department of Higher Algebra, Faculty of Mechanics and Mathematics, Moscow State University, Leninskie Gory, Moscow, 119991, Russia}
\address{Universit\'e Grenoble I, Institut Fourier,
UMR 5582 CNRS-UJF, BP 74, 38402 St. Martin d'H\`eres c\'edex,
France}
\maketitle
\begin{abstract}
Consider an automorphism group of a finite-dimensional %local
algebra.
S.~Halperin conjectured that the unity component of this group is solvable if the algebra is a complete intersection.
The solvability criterion recently obtained by M.~Schulze \cite{Schulze} provides a proof to a local case of this conjecture as well as gives an alternative proof of S.S.--T.~Yau's theorem \cite{Yau} based on a powerful result due to G.~Kempf.
In this note we finish the proof of Halperin's conjecture and study the extremal cases in %of?
Schulze's criterion, where %when?
the algebra of derivations is non-solvable. This allows us to reduce a direct, self-contained %indicates that the proof is completely contained in this note.
proof of Yau's theorem.

%The criterion on the solvability of the unity component of this group recently obtained by M.~Schulze \cite{Schulze} allowed him to give an alternative proof of Yau's theorem based on a powerful result due to G.~Kempf. In this note we study the extremal cases in Schulze's criterion, where the algebra of derivations is non-solvable. This allows us to reduce a direct, self-contained proof of Yau's theorem.
\end{abstract}

\section{Introduction}
Let $\K$ be an algebraically closed field of characteristic zero.
We denote by $R$ the algebra of formal power series $\K[\![x_1,\ldots,x_n]\!]$ % or a polynomial algebra $\K[x_1,\ldots,x_n]$.
and by $\mm$ the maximal ideal $(x_1,\ldots,x_n)\lhd R$. Let $I\subset\mm$ % для единообразия: везде использовать \subset или \subseteq?
be such that $S=R/I$ is a finite-dimensional (or Artin) %как лучше артиновость упомянуть?
local algebra with the maximal ideal $\m=\mm/I$.

Consider the automorphism group $\Aut S$. This is an affine algebraic group with the tangent algebra being the Lie algebra of derivations $\Der S$; see \cite[Ch.1, \textsection2.3, ex. 2]{ViOn}. %точна ли ссылка
So the solvability of the connected component of unity $(\Aut S)^\circ$ (or \emph{unity component} for short) is equivalent to the solvability of the Lie algebra $\Der  S$.
%TODO вставить про теорему Яу из abstract.
%This question of solvability is the main problem examined in this paper.

In 2009 M.~Schulze obtained the following criterion, which has several applications discussed below.

\begin{theo}[Schulze, \cite{Schulze}]\label{schulze}
Let $S=R/I$ be a finite-dimensional local algebra, where $I\subseteq \mm^l$. If the inequality
\begin{equation}\label{Schulze-inequality}
  \dim(I/\mm I)<n+l-1
\end{equation} holds, then the algebra of derivations $\Der S$ is solvable.
\end{theo}

This article provides a generalization of that criterion for a non-local case, see Corollary~\ref{schulze global}, as well as presents a new criterion based on similar techniques, see Theorem \ref{mycriterion}. These two criteria work for different types of algebras. %later?

To mention some applications of Schulze's criterion % in the global setting? надо сказать где-то, что гипотеза Гальперина нелокальна
let us consider a \emph{regular sequence} $f_1,\ldots,f_n\in\K[x_1,\ldots,x_n]$\footnote{i.e. the image of $f_i$ in the quotient $\K[x_1,\ldots,x_n]/(f_1,\ldots,f_{i-1})$ is not a zero divisor for all $i$.}. Equivalently, the quotient $S=\K[x_1,\ldots,x_n]/(f_1,\ldots,f_{n})$ is non-trivial and finite-dimensional and is called a global \emph{complete intersection}. % global complete intersection?

\begin{conjecture}[Halperin, 1987\footnote{This conjecture was proposed by S. Halperin at the conference in honor of J.--L.~Koszul.}]\label{Halperin}
%  Let $f_1,\ldots,f_n\in\CC[x_1,\ldots,x_n]$ be a regular sequence.
Suppose that a finite-dimensional algebra $S$ is a global complete intersection.
Then the unity component $(\Aut S)^\circ$ of the automorphism group of $S$ is solvable.
\end{conjecture}

 Just a few months later H.~Kraft and C.~Procesi proved the conjecture in the case of homogeneous polynomials. %It is easy to see that in this case the quotient algebra is local.
\begin{theo}[Kraft--Procesi, \cite{KrPro}]\label{KrPro}
   Assuming $\K=\CC$, let $f_1,\ldots,f_n\in \K[x_1,\ldots,x_n]$ be homogeneous polynomials, and the algebra
   \begin{equation}
     S=\K[x_1,\ldots,x_n]/(f_1,\ldots,f_n)
   \end{equation} be finite-dimensional. Then the unity component $(\Aut S)^\circ$ is solvable.
\end{theo}
%Meanwhile, in the general case the conjecture remains open. Nevertheless, in the local case it turns out to be a direct consequence of Schulze's Theorem \ref{schulze}.
In this case the algebra $S$ is local. Thereby, the generalization of Theorem~\ref{KrPro} turns out to be a direct consequence of Schulze's Theorem \ref{schulze} as follows.

\begin{corollary}[Schulze, {\cite[Corollary 2]{Schulze}}]\label{intersect}
 Given a local complete intersection $S=R/(f_1,\ldots,f_n)$, the group $(\Aut S)^\circ$ is solvable.
\end{corollary}
\begin{proof}
We may suppose that $f_i\in\mm^2$. Then $\dim(I/\mm I)=n$ and so (\ref{Schulze-inequality}) is fulfilled.
\end{proof}

%Thus, to finish the proof of Conjecture~\ref{Halperin} it suffices to deduce the global case from the local one. We do it in Section \ref{Halperin_section}.
In Section \ref{Halperin_section} we introduce a strict criterion for the solvability of the algebra of derivations $\Der S$ for a non-local finite-dimensional algebra $S$, see Theorem~\ref{global criterion}. This allows us to deduce the global case of Conjecture~\ref{Halperin} from the local one, thus to finish its proof.
%Thus, the hypothesis of S.~Halperin is proved in the case of local factor algebra $\K[x_1,\ldots,x_n]/(f_1,\ldots,f_n)$.  %(see Corollary \ref{generators}).
 %Notice that this algebra is local (and coincides with $\K[\![x_1,\ldots,x_n]\!]/(f_1,\ldots,f_n)$) if and only if the intersection of hypersurfaces $\mathbb{V}(f_i)$ consists of a single point.

Now let us consider isolated hypersurface singularities (or IHS, for short).
Let $p\in\K[x_1,\ldots,x_n]$ % не ряд?
 be such that the hypersurface $\{p=0\}\subset\K^n$ has an isolated singularity $H=(\{p=0\},0)$ at the origin.
Let $J(p)=\left(\frac{\partial p}{\partial x_1},\ldots,\frac{\partial p}{\partial x_n}\right)$ be a \emph{Jacobian ideal} of $p$.
The quotient $A(H)=\K[\![x_1,\ldots,x_n]\!]/(p,J(p))$ is called a \emph{local algebra} or a \emph{moduli algebra} of the IHS $H$.

  Since the formal power series ring is local, the algebra $A(H)$ is local as well. There is an analogue of the Hilbert Nullstellensatz for the germs of analytic functions (called \emph{R\"uckert Nullstellensatz}, see \cite[Theorem 3.4.4]{Jong}, \cite[30.12]{abhya}, \cite{Ruckert}), which holds for formal power series as well, since there is a purely algebraic proof. So, the radical $\sqrt{(p,J(p))}$ coincides with the maximal ideal $\mm$. Thus, the ideal $(p,J(p))$ contains some degree of the maximal ideal or, equivalently, the algebra $A(H)$ is finite-dimensional. Conversely, if the algebra $A(H)$ is finite-dimensional then the singularity $H$ is isolated. Indeed, the finite dimensionality of $A$ is equivalent to inclusion $\mm^r\subseteq I$ for some $r$, or $\sqrt{(p,J(p))}=\mm$. It implies $\mathbb{V}(p,J(p))=0$, and $H$ is the only singularity in some neighbourhood of zero.

It has been proven by J. Mather and S.~S.--T.~Yau in \cite{Mather-Yau} that two IHS are biholomophically equivalent if and only if their moduli algebras are isomorphic. Thus, the finite-dimensional local algebra $A(H)$ defines the IHS $H$ up to an analytic isomorphism.

In order to determine when a finite-dimensional local algebra is a moduli algebra of some IHS, S.S.--T.~Yau \cite{Ya1} introduced a Lie algebra of derivations $L(H)=\Der A(H)$ called sometimes a \emph{Yau algebra} and obtained the following result.
\begin{theo}[Yau,\cite{Yau}]\label{yau}
The algebra $L(H)$ is solvable.
\end{theo}

Note that generally the Yau algebra does not uniquely determine its moduli algebra. But for \emph{simple} singularities this property holds with only one exception. Their classification is well known and consists of two infinite series $A_k, D_k$ and three exceptional singularities $E_6, E_7, E_8$; e.g. see \cite[Chapter 2]{ArnGV}.
A. Elashvili and G.~Khimshiashvili %in their study of simple singularities
proved the following fact.
\begin{theo}[Elashvili--Khimshiashvili, {\cite[Theorem 3.1]{Elash}}]
  Let $H_1$ and $H_2$ be two simple IHS, except the pair $A_6$ and $D_5$. Then $L(H_1)\cong L(H_2)$ if and only if $H_1$ and $H_2$ are analytically
   isomorphic.
\end{theo}

\begin{remark}\label{quasihom}
  Assume the polynomial $p$ is quasi-homogeneous, i.e.
  \begin{equation}
   p(\lambda^{k_1}x_1,\ldots,\lambda^{k_n}x_n)=\lambda^k p(x_1,\ldots,x_n)\mbox{ for some fixed }k,k_1,\ldots,k_n\in\NN.
  \end{equation}
     Then $p\in J(p)$ and the moduli algebra $\K[\![x_1,\ldots,x_n]\!]/(p,J(p))$ is a complete intersection. Under this assumption Theorem \ref{yau} is a particular case of Corollary \ref{intersect}.
\end{remark}

In \cite{Schulze} M.~Schulze deduces Theorem \ref{yau} from his criterion.
In order to prove it he uses the following deep result of G.~Kempf.
\begin{theo}[Kempf, {\cite[Theorem 13]{Kempf}}]
  Let $p$ be a homogeneous polynomial of degree $d\ge3$ defined as a regular function on the space $\CC^n$ endowed with a linear action of a semisimple group $G$.
  If the Jacobian $J(p)$ is a $G$-invariant subspace then there exists such a $G$-invariant polynomial $q$ that $J(p)=J(q)$.
\end{theo}
%M. Schulze poses a question if Theorem \ref{yau} can be obtained without the Kemph result.

 \begin{definition}
Let us call the finite-dimensional local algebra $S$ an \emph{extremal algebra} if the equality $\dim{I/\mm I}=l+n-1$ holds.
\end{definition}
The description of the extremal algebras with a non-solvable algebra of derivations allows to deduce Theorem~\ref{yau} from Schulze's criterion without using the Kemph result, as explained in Section \ref{extremal algebras}.

\begin{definition}%REDO !!
  Let us say that a graded local finite-dimensional algebra $S=R/I$ is \emph{narrow} if there holds an inequality
  \begin{equation}
    \dim I_k-\dim (\m I)_k\le k\mbox{ for all } k=1,2\ldots,
  \end{equation}
  where $J_k$ is the $k$th graded component of a graded ideal $J$. In other words, there exists such a set of homogeneous generators of $I$ that the number of generators of degree $k$ is not greater than $k$ for each $k$.
\end{definition}
\begin{remark}
  If $I\subseteq\mm^r$ then for an algebra $S=R/I$ to be narrow it is sufficient to check the inequality for $k\le r,\,$ since $I_k=(\mm I)_k$ for $k>r$.
\end{remark}

Recall that
an \emph{associated graded algebra} of the local algebra $S$ is the algebra $\gr S=\K\oplus (\m/\m^2)\oplus (\m^2/\m^3)\oplus\ldots$, i.e. $(\gr S)_i=\m^i/\m^{i+1}$. Now introduce a solvability criterion as follows.

\begin{theo}\label{mycriterion}
  Suppose that the associated graded algebra $\gr S$ of a local finite-dimensional algebra $S$ is narrow. Then the algebra of derivations $\Der S$ is solvable.
\end{theo}

The proof is given below. Finally, in the last section we give a lower bound for the dimension of the automorphism group and obtain an algebra with the unipotent automorphism group.

Let us mention a related result on solvability of the group of equivariant automorphisms. Consider a connected affine algebraic group $G$ and an irreducible affine $G$-variety $X$. Assume that the number of $G$-orbits on $X$ is finite and $X$ contains a $G$-fixed point. Then the unity component $(\Aut_G X)^\circ$ of the group of $G$-equivariant automorphisms of the variety $X$ is solvable; see \cite[Theorem~1]{ArTi}.

\section{Solvability criteria}\label{analysis}
In this section we provide a simplified proof of Theorem~\ref{schulze} and introduce a new solvability criterion. %REDO

   If $I\supseteq \mm^r\lhd\K[x_1,\ldots,x_n]$ for some $r$ then $R/\widetilde{I}\cong\K[x_1,\ldots,x_n]/I$, where $\widetilde{I}$ is an ideal in the algebra of formal power series generated by $I$. Therefore it makes no difference whether the local algebra is obtained by factorization of the algebra of polynomials or the algebra of formal power series.

  \begin{proposition}\label{genW}
    Suppose that the ideal $I\lhd R$ is represented in the form $I=W\oplus\mm I$. Then the ideal $(W)$ generated by subspace $W$ coincides with $I$.
  \end{proposition}
  \begin{proof}
    Consider the factorization mapping $\varphi\colon R\to R/(W)$. The quotient algebra is a local algebra with the maximal ideal $\varphi(\mm)$. The decomposition $I=W\oplus\mm I$ implies $\varphi(I)=\varphi(\mm I)$. Since the ring $R$ is Noetherian, the ideals $I$ and $\varphi(I)$ are finitely generated. Then by Nakayama's Lemma (see \cite[Proposition 2.6]{AtiM}) there holds $\varphi(I)=0$, i.e. $(W)=I$. % REDO (see -> ;see
  \end{proof}
\begin{corollary}\label{generators}
  The minimal number of generators of the ideal $I$ is equal to $\dim W=\dim(I/\mm I)$.
\end{corollary}

   Note that Proposition \ref{genW} does not hold for the algebra of polynomials. For example take an ideal $I=\mm^2\lhd\K[x]$ and decompose it as follows,
   \begin{equation}
     \mm^2=\bangle{x^2-x^3}\oplus\mm^3.
   \end{equation}
   It is easy to see that the ideal $(x^2-x^3)$ does not coincide with $\mm^2$.

Below we follow \cite{Schulze} with some improvements.

As always we suppose that $S=R/I$, where $\mm^l\supset I\lhd R=\K[\![x_1,\ldots,x_n]\!]$ for $l\ge2$.
Assume that the algebra $\Der S$ is not solvable. Hence it contains an $\ssl_2$-triple $\{e,f,h\}$ with relations $[e,f]=h$, $[h,f]=-2f$, $[h,e]=2e$.
Note that the automorphisms of $S$ preserve the maximal ideal $\m\lhd S$ as it is unique. All powers of the maximal ideal are preserved %by automorphisms
as well.
Therefore the ideals $\m$ and $\m^2$ are $\ssl_2$-submodules. Since the representations of $\ssl_2$ are completely reducible, the ideal $\m$ contains such an $\ssl_2$-submodule $\overline{V}$ that
\begin{equation}
 \m=\overline{V}\oplus {\m}^2.
\end{equation}
Denote by $\vphi\colon R\to S$ the factorization by the ideal $I$. Since $\vphi(\mm^2)=\m^2$ there exists a subspace $V\subset \mm$ %in algebra $R$
 such that $\mm=V\oplus {\mm}^2$ and $\vphi\colon V\overset{\sim}{\rightarrow} \overline{V}$.
Thus, according to Proposition~\ref{genW} %Lemma \ref{generate}
the subspace $V$ generates the ideal $\mm$ and hence the algebra $R$.
So we may assume up to the change of coordinates that $V=\bangle{x_1,\ldots,x_n}$ and $\overline{V}=\bangle{\x_1,\ldots,\x_n}$, where $\x_i=\vphi(x_i)$.

%The $\ssl_2$-representation is linear on $V$, hence
We may introduce an $\ssl_2$-representation on $V$ by the given isomorphism and extend it to $R$. Note that the factorization map $\vphi$ is a homomorphism of $\ssl_2$-modules.
%Indeed, apply an arbitrary $g\in\ssl_2$ to a monomial $x_{i_1}\cdot\ldots\cdot x_{i_s}\in R$. Then
%\begin{multline}
%\vphi(g\circ(x_{i_1}\ldots x_{i_s}))=\vphi(\sum_{j=1}^s x_{i_1}\ldots(g\circ x_{i_j})\ldots x_{i_s})=\\
%\sum_{j=1}^s \vphi(x_{i_1})\ldots\vphi(g\circ x_{i_j})\ldots \vphi(x_{i_s})=g\circ(\x_{i_1}\ldots \x_{i_s}).
%\end{multline}
Therefore the ideal $I\lhd R$ is an invariant subspace of the $\ssl_2$-representation on $R$. %REDO

%Denote by $\wt(z)$ a weight of the  an eigenvector $z$ of the operator $h$ with the eigenvalue $\wt(z)$ by
Given a \emph{weight vector} $z$, i.e. an eigenvector of the operator $h\in\ssl_2$, denote its weight by $\wt(z)\in\ZZ$.
We may suppose that $x_1,\ldots,x_n$ are the weight vectors of the $\ssl_2$-module $V$, and $x_1,\ldots,x_k$, $k\le n$, are the highest weight vectors with weights $n_i=\wt(x_i)$, where $n_1\ge\ldots\ge n_k\ge0,\;\sum (n_i+1)=n$. %In other words, $\Ker e=\bangle{\x_1,\ldots,\x_k}$.
Denote $V_{high}\coloneq\bangle{x_1,\ldots,x_k},\: V_{rest}\coloneq\bangle{x_{k+1},\ldots,x_n}$.

The ideal $\mm I\subset R$ is $\ssl_2$-invariant by the Leibniz rule, hence $I$ contains the complementary $\ssl_2$-submodule $W$ such that $I=W\oplus\mm I$. By Corollary~\ref{generators} its basis is a minimal set generating $I$.

Similarly to $V=V_{high}\oplus V_{rest}$ consider the decomposition
\begin{equation}
  W=W_{high}\oplus W_{rest}
\end{equation} into the subspace  $W_{high}=\bangle{w_1,\ldots,w_s}$, where $w_i$ are the highest weight vectors of $W$, %the simple $\ssl_2$-submodules,
 and the subspace $W_{rest}$ of the remaining weight vectors of $W$. Notice that $W_{rest}\subset \Im f\subseteq (x_{k+1},\ldots,x_n)$ since $\Im f$ is spanned by weight vectors which are not of highest weight. %REDO

 Let $\varphi_i\colon R\to R/J_i$ be the factorization by the ideal $J_i=(x_{i+1},\ldots,x_{n}),\quad i=1\ldots,k$. Since $J_i\supseteq(x_{k+1},\ldots,x_n)\supset W_{rest}$ the equality $W_i\coloneq\varphi_i(W)=\varphi_i(W_{high})$ holds. Note that $\dim W_i\ge i$, because $\K[\![x_1,\ldots,x_i]\!]/(W_i)\cong R/(J_i,W_i)\cong S/(\x_{i+1},\ldots,\x_{n})$ is finite-dimensional. In particular, $s\ge k$.

By induction we can reorder the highest weight vectors $w_{1},\ldots,w_{s}\in W_{high}$ so that $\varphi_i(w_{1}),\ldots,\varphi_i(w_{i})$ become linearly independent in $W_i$ for all $i$. Then $\wt(w_{i})\ge l n_i$ since $w_i$ contains the monomials in variables $\x_{1},\ldots,\x_i$ of degree at least $l$ and $\wt(x_j)=n_j\ge n_i$ for $j\le i$.

\begin{proof}[Proof of Theorem \ref{schulze}]
The deduction above implies that the subspace $V$ contains a non-trivial $\ssl_2$-submodule and that $n_1>0$. We have
  \begin{multline}\label{ineqs}
  \dim I/\mm I= \dim W \ge\sum_{i=1}^k (l n_i+1)  =(n_1-1)l+l+1+\sum_{i=2}^k (l n_i+1)\ge\\
   (n_1-1)+l+1+\sum_{i=2}^k (n_i+1) =\sum_{i=1}^k(n_i+1)+l-1=n+l-1.
  \end{multline}
 Thus, Theorem \ref{schulze} is proved.
\end{proof}

\begin{proposition}\label{unipotent kernel}
 There exists a natural mapping $\vphi\colon\Aut S\to\Aut (\gr S)$ with a unipotent kernel. %, and the kernel of $\vphi$ is a unipotent group.
\end{proposition}
\begin{proof}
The ideals $\m^i$ are invariant under $\Aut S$ for all $i$, since they are powers of the unique maximal ideal. Therefore, $\Aut S$ naturally acts on $\m^i/\m^{i+1}$ for all $i$, hence it acts on $\gr S$. We obtain a natural map $\vphi\colon\Aut S\to\Aut (\gr S)$.

  Take a basis of $S$ which is consistent with the chain of subspaces $0\subseteq\m^r\subset\ldots\subset \m \subset S$.
  Consider an arbitrary operator $g\in\ker \vphi$. Then $g(z)\in z+\m^{i+1}$ for any $z\in\m^i$, and $g$ is represented by a unitriangular matrix in the taken basis. Hence $\ker\vphi$ is unipotent.
\end{proof}
\begin{corollary}\label{graded-nongraded}
 If the unity component $(\Aut(\gr S))^\circ$ is solvable then the unity component $(\Aut S)^\circ$ is solvable as well.
\end{corollary}

\begin{theo}\label{dimk}
  The algebra of derivations $\Der S$ of a narrow algebra $S$ is solvable.
\end{theo}
\begin{proof}
    Let $\Der S$ be non-solvable. Then the deduction above is applicable. Consider a simple $\ssl_2$-submodule $F=\ssl_2\cdot w_1\subseteq I$. It has a zero intersection with the ideal $\mm I$. Let $k$ be the biggest integer such that $F\subseteq \mm^k$. Then $F$ has a zero intersection with $\mm^{k+1}$ as well and the highest weight of $F$ is equal to $k n_1\ge k$. After factorization by $\mm^{k+1}$ it implies that $\dim I_k\ge \dim(\mm I)_k+\dim F>\dim(\mm I)_k+k$, and the algebra $S$ is not narrow.
\end{proof}
\begin{proof}[Proof of Theorem \ref{mycriterion}]
  The desired statement is a direct consequence of Theorem \ref{dimk} and Corollary \ref{graded-nongraded}.
\end{proof}
\begin{remark}
  As a matter of fact, Theorem \ref{mycriterion} can be obtained without Proposition \ref{unipotent kernel}. However, then the proof loses in clarity.
\end{remark}
\begin{example}
Algebras \begin{align}
  A=&\K[x,y]/(x^2, y^3,xy^2),\\
  B=&\K[x,y,z]/(x^3,x^2y,x^2z,y^4,z^4)
\end{align} are extremal and Schulze's criterion is not applicable, but their algebras of derivations are solvable due to Theorem \ref{mycriterion}.
Actually, the complete automorphism groups of $A$ and $B$ are solvable as well. For example, we obtain through a direct calculation
$$\Aut A=\left\{\left.
\begin{aligned}
  \x&\mapsto c_1\x+a_2\x\y+a_3\y^2,\\
  \y&\mapsto c_2\y+a_4 \x+a_5 \x\y+a_6 \y^2
\end{aligned}
\right|a_i\in\K, c_i\in\K^\times\right\}.$$
\end{example}
\begin{example}
  On the contrary, for the algebra
  \begin{equation}
    A=\K[x_1,\ldots,x_n]/(x_1^l,x_2^l,\ldots,x_n^l), \mbox{ where } l<n,
  \end{equation}
  Schulze's criterion holds but the criterion of Theorem \ref{mycriterion} does not. Note that for $n\ge 5$ the group $\Aut A$ is non-solvable, as far as it contains the subgroup of permutations of coordinates.
\end{example}
Thereby, these two criteria have distinct areas of application.

\section{Extremal algebras and Yau's theorem}\label{extremal algebras}
% As above we have $S=R/I$, where $I\lhd R=\K[\![x_1,\ldots,x_n]\!]$ and $\mm^r\subseteq I\subseteq\mm^l$.
Recall that by an extremal algebra we mean the finite-dimensional algebra satisfying the equality $\dim{I/\mm I}=l+n-1$ in terms of Theorem~\ref{schulze}.
\begin{theo}\label{extrem}
 An extremal algebra $S$ has a non-solvable algebra of derivations $\Der S$ if and only if it is of the form $S=S_1\otimes S_2$, where
  \begin{align}
   S_1&\cong\K[\![x_1,x_2]\!]/(x_1^l,x_1^{l-1}x_2,\ldots,x_1x_2^{l-1},x_2^l)\mbox{ for some } l\ge2,\\
   S_2&\cong\K[\![x_3,\ldots,x_{n}]\!]/(w_2,\ldots,w_{n-1}),
 \end{align}
  and $w_i\in \mm^l\cap\K[\![x_3,\ldots,x_{n}]\!]$ form a regular sequence.
\end{theo}
\begin{proof}
  Suppose that $S=S_1\otimes S_2$ as above. Then the group $\GL(\bangle{x_1,x_2})$ may be embedded into $\Aut S_1$, hence the subalgebra $S_1$ carries a natural $\ssl_2$-representation. We suppose this representation to be trivial on $S_2$.

  Vice-versa, let the algebra of derivations $\Der S$ of the local algebra $S=R/I$ be non-solvable. Recall the deduction from Section~\ref{analysis}. Then $S$ is extremal if and only if the equalities hold in the chain of inequalities (\ref{ineqs}). The first equality holds if and only if $W$ contains exactly $k$ simple $\ssl_2$-submodules, and their weights are $l n_1,\ldots,l n_k$. The second inequality holds if and only if $n_1=1,\; n_2=\ldots=n_k=0$.

  Under these circumstances $k=n-1$, and the simple $\ssl_2$-submodules of $\overline{V}$ are $\bangle{\x_1,\x_2}, \bangle{\x_3},\ldots,\bangle{\x_{n}}$.
Then $W_{high}=\bangle{w_1,\ldots,w_{n-1}}$, where $\wt(w_1)=l,\;\wt(w_i)=0$ for $i=2,\ldots,n-1$. We have
  \begin{equation}\label{extreq}
    S=R/(w_1,f\cdot w_1,\ldots,f^l\cdot w_1,w_2,\ldots,w_{n-1}).
  \end{equation}
      Note that the algebra $\ssl_2$ annihilates the series $w_2,\ldots,w_{n-1}$, hence they do not depend on $x_1$ and $x_2$ and belong to $\K[\![x_3,\ldots,x_{n}]\!]\cap\mm^l$.
      Since $w_1$ is the highest vector of weight $l$ it is of the form $x_1^l g$, where $g\in\K[\![x_3,\ldots,x_{n}]\!]$. Then $f^k\cdot w_1=x_1^{l-k}x_2^k g$.

   Taking into accordance $x_1^r\in I$ the equality
   \begin{equation}
     x_1^r=p_0 x_1^l g+ p_1 x_1^{l-1}x_2 g+\ldots+p_l x_2^l g+q_2 w_2+\ldots+q_{n-1}w_{n-1}
   \end{equation}
   holds for certain $p_i,q_j\in R$. If we substitute $x_1$ for $x_2$ %instead of $x_2$
   on the right side of equation, we get
    \begin{equation}
     x_1^r=(\widetilde{p}_0+\ldots+\widetilde{p}_l) x_1^l g+\widetilde{q}_2 w_2+\ldots+\widetilde{q}_{n-1}w_{n-1},
   \end{equation}
    where $\widetilde{p}_i=p_i(x_1,x_1,x_3,\ldots,x_n),\,\widetilde{q}_j=q_j(x_1,x_1,x_3,\ldots,x_n)$. Clearly, the series $\widetilde{q}_j$ and $\widetilde{p}=\sum_{i=0}^l \widetilde{p}_i$ may be assumed homogeneous in $x_1$. Then $\widetilde{p}=x_1^{r-l}\widehat{p}$, $\widetilde{q}_j=x_1^r\widehat{q}_j$, where $\widehat{p},\widehat{q}_j\in\K[\![x_3,\ldots,x_{n}]\!]$. But it implies
   \begin{equation}
    x_1^l=\widehat{p}x_1^l g+ x_1^l\widehat{q}_2 w_2+\ldots+ x_1^l\widehat{q}_{n-1} w_{n-1}.
   \end{equation}

   Thus me can take the $\ssl_2$-submodule $\bangle{x_1^l,x_1^{l-1}x_2,\ldots,x_1x_2^{l-1},x_2^l}$ instead of
    the $\ssl_2$-submodule $\bangle{w_1,f\cdot w_1,\ldots,f^l\cdot w_1}$ in (\ref{extreq}).

   Finally, the algebra $S$ decomposes into the tensor product $S_1\otimes S_2$, where $S_i$ are as required.
\end{proof}

% \begin{example}
%According to Theorem \ref{extrem}  the algebra of derivations of the extremal algebra
%\begin{equation}
%  S=\K[x_1,\ldots,x_n]/(x_1^l,x_1^{l-1}x_2,x_1^{l-2}x_2^2,\ldots,x_2^l;x_3^{l_3},\ldots,x_{n}^{l_{n}}),\mbox{ where } l_i\ge l\ge2,
%\end{equation} is not solvable.
% \end{example}

   Now let us introduce the following well-known technical lemma for power series, e.g. see \cite[Section 11.1]{ArnGV}. % верна ли ссылка?
  For an arbitrary power series $g$ denote by $g_{(k)}$ its $k$th homogeneous component. %of degree $k$.

\begin{lemma}\label{x2+q}
Suppose $p\in\mm^2\setminus\mm^3\subset R$. Then up to the analytical change of coordinates $p=x_1^2+\ldots+x_k^2+q(x_{k+1},\ldots,x_n)$, where $q\in\mm^3\cap\K[\![x_{k+1},\ldots,x_{n}]\!]$.
\end{lemma}
\begin{proof}
The homogeneous component $p_{(2)}$ is a quadratic form, hence it is of the form $x_1^2+\ldots+x_k^2$ up to the linear change of coordinates.
 Consider a decomposition in $x_1$ as follows, 
 \begin{equation}
 p=a_0+a_1 x_1+a_2 x_1^2+\ldots,
 \end{equation}
  where $a_i\in\K[\![x_2,\ldots,x_n]\!]$ and $a_2$ is invertible.

 Now consider a change of coordinates $\varphi:x_1\mapsto x_1+g,\,x_2\mapsto x_2,\,\ldots,\,x_n\mapsto x_n$, where $g\in \mm\cap\K[\![x_2,\ldots,x_n]\!]$.
Thus,
 \begin{equation}
   \varphi(p)=\widetilde{a}_0+\widetilde{a}_1 x_1+\widetilde{a}_2 x_1^2+\ldots
 \end{equation}
 for some $\widetilde{a}_i\in\K[\![x_2,\ldots,x_n]\!]$. In this case
 \begin{equation}\label{a_1}
   \widetilde{a}_1=a_1+2 g a_2+3g^2a_3+\ldots=a_2(\frac{a_1}{a_2}+2g+3g^2\frac{a_3}{a_2}+\ldots).
 \end{equation}
 Since $a_1\in\mm$, we can choose such $g$ that $\widetilde{a}_1=0$. Indeed, let us divide a series in parentheses on the right side of equation (\ref{a_1}) into homogeneous components. They are of the form $2g_{(k)}+P_k$, where $P_k$ depends only on first $k-1$ homogeneous components of $g$. Thus, by induction all the homogeneous components of the series $g$ are uniquely defined. Note that $\widetilde{a}_2=a_2+3a_3g+ 6a_4 g^2+\ldots$ is still invertible.

 Therefore we may suppose that $a_1=0$. Consider a change $x_1\mapsto x_1 b$ where $b^2=a_2^{-1}$ (other coordinates are untouched). It is easy to see that the required series $b$ exists. It is a change of coordinates since $b$ is invertible. So, we may as well suppose $a_2=1$.

 Finally, consider a change $x_1\mapsto x_1+b_2 x_1^2+b_3 x_1^3+\ldots$. Similarly to the choice of $g$, we may take such $b_i$ that $p\mapsto a_0+x_1^2$. By induction by $n$ we may suppose that $a_0(x_2,\ldots,x_n)$ is of a required form. Then $p$ is of a required form as well.
\end{proof}

%\begin{lemma}\label{tail}
%     Suppose that a series $p$ defines the IHS $H=(\{p=0\},0)$. Then there exists a polynomial $q$ such that the ideals $(p,J(p))$ and $(q,J(q))$ of the algebra of formal power series coincide and the polynomial $q$ defines the IHS that is isomorphic to $H$.
%\end{lemma}
%\begin{proof}
%    There holds $\mm^r\subseteq I\coloneq(p,J(p))$ for some $r$. Consider a polynomial $q$ that is constructed from $p$ by cutting off the terms of  degree exceeding $r+1$. Then $J(p-q)\subseteq\mm^{r+1}$.
%  By $\mm^{r+1}\subseteq \mm I$ we get
%  \begin{equation}
%    I=\bangle{p,J(p)}+\mm I=\bangle{q,J(q)}+\mm I.
%  \end{equation}
%  Thus, by Proposition \ref{genW} the ideal $(q,J(q))$ equals $I$. Therefore, the moduli algebras of the IHS $(\{p=0\},0)$ and $(\{q=0\},0)$ coincide, hence according to \cite{Mather-Yau} these IHS are biholomorphically equivalent.
%\end{proof}

\begin{proof}[Proof of Theorem \ref{yau}]
According to Yau's remark in \cite{Yau} we may suppose $p\in\mm^3$. Indeed, let $p\in\mm^2\setminus\mm^3$. By Lemma \ref{x2+q} we have  $p=x_1^2+\ldots+x_k^2+q(x_{k+1},\ldots,x_n)$ up to the change of coordinates. Since $\frac{\partial p}{\partial x_i}=2x_i$, $i=1,\ldots,k$, the equality $\K[\![x_1,\ldots,x_n]\!]/(p,J(p))\cong\K[\![x_{k+1},\ldots,x_n]\!]/(q,J(q))$ holds, and we may take a series $q\in\mm^3\cap\K[\![x_{k+1},\ldots,x_n]\!]$ instead of $p$. %Iteratively we obtain the required statement. In addition, by Lemma \ref{tail} we may suppose that $p$ is a polynomial.

If $p\in\mm^3$ then $I=(p,J(p))\subset\mm^2$ and $l\ge 2$. We should prove that the Yau algebra $\Der S$ of the moduli algebra $S=\K[\![x_1,\ldots,x_n]\!]/I$ is solvable. Assume the contrary.

Note that $\dim(I/\mm I)\le n+1\le n+l-1$. Hence the moduli algebra either satisfy the inequality of Schulze's criterion and $\Der S$ is solvable, or it is extremal and $l=2$. In the latter case Theorem \ref{extrem} may be applied and $S=S_1\otimes S_2$, where $S_1=\K[\![x_1,x_2]\!]/(x_1^2,x_1x_2,x_2^2)$ and $S_2=\K[\![x_3,\ldots,x_{n}]\!]/(w_2,\ldots,w_{n-1})$.

 It is easy to see that the homogeneous component $p_{(3)}$ has a form $p_1(x_1,x_2)+p_2(x_3,\ldots,x_{n})$, since otherwise $J(p)$ would contain a series with a term of the form $x_ix_j$, where $i\in\{1,2\},j\in\{3,\ldots,{n}\}$.
Therefore
\begin{equation}
  \bangle{x_1^2,x_1x_2,x_2^2}\subseteq(p,J(p))\cap(\K[x_1,x_2])_2\subseteq\left\langle\frac{\partial p_1}{\partial x_1},\frac{\partial p_1}{\partial x_2}\right\rangle,
\end{equation}
a contradiction.
\end{proof}

\section{The global case and Halperin's conjecture}\label{Halperin_section}
Let the algebra $S=\K[x_1,\ldots,x_n]/I$ be a finite-dimensional, not necessarily local algebra. Suppose that it contains $s$ maximal ideals $\m_1,\ldots,\m_s$.
Then there exists an integer $k\in\NN$ such that $\prod_{i=1}^s\m_i^k=0$. Since the ideals $\m_i^k$ are coprime, the equation $\bigcap_{i=1}^s\m_i^k=\prod_{i=1}^s\m_i^k$ holds. Finally, by \cite[Theorem 8.7]{AtiM} the algebra $S$ can be decomposed into a direct product of local subalgebras as follows,
\begin{equation}\label{local-decompose1}
  S\cong\prod_{i=1}^s S/\m_i^k.
\end{equation}
In addition, any decomposition of $S$ into the local subalgebras is of the form (\ref{local-decompose1}), i.e. they are uniquely determined up to isomorphism.

 On the other hand, there is a unique maximal decomposition
 \begin{equation}\label{unity-decompose}
   1=e_1+\ldots+e_t
 \end{equation} of the unity into a sum of orthogonal idempotents, i.e. $e_ie_i=e_i$ for all $i$ and $e_ie_j=0$ for $i\neq j$; e.g. see \cite[Section II.5]{Artin}. Then we have the decomposition
 \begin{equation}\label{local-decompose2}
   S=\bigoplus_{i=1}^t e_i S,
 \end{equation} where the subalgebras $e_i S$ are indecomposable and hence local. It means that
  \begin{equation}
    S_i=e_i S\cong S/\m_i^k
  \end{equation} up to the permutation of indexes and that $t=s$.
Thus, we have a uniquely determined decomposition (\ref{local-decompose2}) of $S$ into the local subalgebras.

 \begin{proposition}\label{aut decomposition}
   $(\Aut S)^\circ=(\Aut S_1)^\circ\times\ldots\times(\Aut S_s)^\circ.$
 \end{proposition}
\begin{proof}
  Since the unity $1$ of algebra $S$ and its decomposition (\ref{unity-decompose}) are unique, the primary idempotents $e_i$ in this decomposition are preserved by $(\Aut S)^\circ$ as well as the subalgebras $S_i$. Since $S_iS_j=0$ for $i\neq j$, the required statement holds.
\end{proof}

   Denote by $\mm_i=(x_1-a_{1i},\ldots,x_n-a_{ni}),\,i=1\ldots s,$ the maximal ideals in $\K[x_1,\ldots,x_n]$ corresponding to $\m_1,\ldots,\m_s\lhd S$. Then $\bigcap_{i=1}^s\mm_i^k\subset I$. Let us introduce the corresponding algebras of formal power series $R_i=\K[\![x_1-a_{1i},\ldots,x_n-a_{ni}]\!]$. %and ideals $I=(I)\lhd R_i$.
%So, $\prod_{i=1}^s\mm_i^k\subset(f_1,\ldots,f_n)$.
 The ideal $(\mm_j)\lhd R_i$ coincides with $R_i$ unless $i=j$. In latter case $(\m_i)\lhd R_i$ is the maximal ideal which we denote by $\widetilde{\mm}_i$. Therefore, the inclusion $\widetilde{\mm}_i^k=\bigcap_{i=1}^s\left(\mm_i^k\right)\subset (I)\lhd R_i$ holds, and
 \begin{equation}
   S_i\cong S/\m_i^k\cong \K[x_1,\ldots,x_n]/(I,\mm_i^k)\cong R_i/(I).
 \end{equation}
 Taking into account Proposition~\ref{aut decomposition}, we obtain the following solvability criterion.
\begin{theo}\label{global criterion}
  The unity component $(\Aut S)^\circ$ for the finite-dimensional algebra $S=\K[x_1,\ldots,x_n]/I$ with maximal ideals $\m_1,\ldots,\m_s$ is solvable if and only if the unity component $(\Aut S_i)^\circ$ for the local algebra $S_i=R_i/(I)$ is solvable for each $i=1\ldots s$.
\end{theo}
Now we can prove Halperin's conjecture.
\begin{proof}[Proof of Conjecture \ref{Halperin}]
Suppose $S=\K[x_1,\ldots,x_n]/(f_1,\ldots,f_n)$.
 Since the local subalgebra $S_i=R_i/(f_1,\ldots,f_n)$ is a local complete intersection, the group $(\Aut S_i)^\circ$ is solvable by Corollary~\ref{intersect} for each $i$. Then by Theorem~\ref{global criterion} the group $(\Aut S)^\circ$ is solvable as well.
\end{proof}

Theorem~\ref{global criterion} allows us to deduce the following globalization of Schulze's criterion.
\begin{corollary}\label{schulze global}
%  Given a finite-dimensional algebra $S=\K[x_1,\ldots,x_n]/I$ and an integer $l>1$, suppose that for any maximal ideal $\mm\subset\K[x_1,\ldots,x_n]$ either $I\nsubseteq\mm$ or $I\subset\mm^l$. Then the unity component $(\Aut S)^\circ$ is solvable.

Suppose the ideal $I\subset \K[x_1,\ldots,x_n]$ with $m$ generators and an integer $l>1$ be such that the following holds.
\begin{itemize}
  \item The quotient algebra $S=\K[x_1,\ldots,x_n]/I$ is finite-dimensional.
  \item For any maximal ideal $\mm\subset\K[x_1,\ldots,x_n]$ there holds either $I\nsubseteq\mm$ or $I\subset\mm^l$.
  \item An inequality $m<n+l-1$ holds.
\end{itemize}
Then the unity component $(\Aut S)^\circ$ is solvable.
\end{corollary}
\begin{proof}
  Let $\mm_1,\ldots,\mm_s\lhd \K[x_1,\ldots,x_n]$ be the only ideals containing $I$, as above. By Corollary~\ref{generators} there holds $\dim(I/\mm_i I)\le m < n+l-1$ and Schulze's criterion is applicable for the algebras $R_i/(I)$, $i=1,\ldots,s$. Therefore $(\Aut S)^\circ$ is solvable by Theorem~\ref{global criterion}.
%   By Theorem~\ref{global criterion} it is sufficient to prove that Schulze's criterion holds for the algebras $R_i/(I)$, $i=1,\ldots,s$. But $\dim(I/\mm_i I)\le m < n+l-1$ by Corollary~\ref{generators}. %the minimal number of generators of $(I)\lhd R_i$ equals $\dim(I/\mm_i I)$, hence $\dim(I/\mm_i I)\le m < n+l-1$ and Schulze's criterion is applicable. % REDO style
\end{proof}

\section{Automorphism subgroups and dimension bounds} As usual we assume that the ideal $I\lhd R$ contains $\mm^l$, where $l\ge 2$, and the algebra $S=R/I$ is finite-dimensional and local with the maximal ideal $\m=(\x_1,\ldots,\x_n)$.

Recall that the sum of all minimal ideals of a finite-dimensional algebra $S$ is called a \emph{socle} $\Soc S$. It is invariant under endomorphisms of $S$.
   An \emph{annihilator} of an arbitrary subset $X\subset S$ is the ideal $\Ann X=\{z\in S\;|\; zX=0\}$.
 \begin{lemma}\label{ann}
   $\Soc S= \Ann \m.$
 \end{lemma}
\begin{proof}
  Consider an arbitrary minimal ideal $J\subseteq \Soc S$. Obviously, $\m J\subseteq J$. But $\m J\neq J$ by Nakayama's Lemma. %, otherwise $J=\m J=\ldots=\m^r J=0$.
  Thus, $\m J=0$ and $\Soc S\subseteq \Ann \m$.

  Suppose $z\in\Ann \m$. Then $zS=\{z(c+w)\;|\; c\in\K,w\in\m\}=\{cz\;|\;c\in\K\}$, so the principal ideal $(z)$ is one-dimensional and minimal. It implies $\Ann\m\subseteq\Soc S$.
\end{proof}

   Assuming that $S$ is graded, Y.-J. Xu and S. S.-T. Yau found a dimension bound for the group $\Aut S$ as follows,
   \begin{equation}
     \dim \Aut S\ge \dim S-\dim\Soc S,
   \end{equation}
   see \cite[Proposition 2.3]{Xu-Yau}.
  In Theorem~\ref{dum aut} we introduce a lower bound without this assumption. % for the dimension of the group $\Aut S$. %Note that the corresponding unipotent subgroup may not exist.
\begin{definition}
   Let us call a \emph{lower socle} of the algebra $S$ the ideal $\LSoc S=\Soc S\cap \m^2$. We may choose a subspace $\USoc S\subseteq\Soc S$ such that
   \begin{equation}
      \Soc S=\USoc S\oplus\LSoc S.
   \end{equation} Let us call it an \emph{upper socle}. Note that the choice of the upper socle is not canonical.
However, up to the change of coordinates we may suppose $\USoc S\subseteq\bangle{\x_1,\ldots,\x_n}$.
 \end{definition}
\begin{proposition}\label{Udim}
  The automorphism group of a finite-dimensional local algebra $S$ contains a unipotent subgroup $U\subseteq \Aut S$ with
 \begin{multline}
%  \dim U= \dim (\m/\m^2)\cdot\dim(\Soc S\cap\m^2)+(\dim (\m/\m^2)-(\dim(\Soc S)-\dim(\Soc S\cap\m^2)))\cdot(\dim(\Soc S)-\dim(\Soc S\cap\m^2)).
    \dim U=\dim(\LSoc S)\cdot\dim(\m/\m^2) + \dim(\USoc S)\cdot(\dim(\m/\m^2)-\dim(\USoc S)).
 \end{multline}
\end{proposition}
\begin{proof}
Suppose that $\USoc S=\bangle{\x_1,\ldots,\x_s}$.
Consider the unipotent subgroup of linear transformations
\begin{multline}
  U=\{u\colon \x_1\mapsto\x_1+F_1,\ldots,\x_n\mapsto\x_n+F_n\,|\\
  \,F_1,\ldots,F_s\in\LSoc S,\,F_{s+1},\ldots,F_n\in\Soc S\}\subset\GL(S),
\end{multline}
acting trivially on a subspace $\bangle{1}\oplus\m^2$.
%The affine subspace $U$ in $\GL(S)$ forms an abelian group relatively to composition. The unity is the element $(F_1=0,\ldots,F_n=0)$, an inverse to $(F_1,\ldots,F_n)$ is $(-F_1,,\ldots,-F_n)$.
It is easy to see that
\begin{equation}
  u(\x_i) u(\x_j)=(\x_i+F_i)(\x_j+F_j)=\x_i\x_j=u(\x_i\x_j)\mbox{ for }i,j\in\{1,\ldots,n\}, u\in U.
\end{equation}
Hence
\begin{equation}
  u(a) u(b)=ab=u(ab)\mbox{ for }a,b\in\m, u\in U.
\end{equation}
Therefore the $U$-action is consistent with the multiplication in $S$, so $U\subseteq\Aut S$.
Finally,
\begin{multline}
  \dim U= s\cdot\dim(\Soc S)+ (n-s)\cdot\dim(\LSoc S)=\\
  s\cdot\dim(\USoc S)+ n\cdot\dim(\LSoc S).
\end{multline}
\end{proof}
\begin{theo}\label{dum aut}
    $\dim\Aut S\ge \dim (\m/\m^2)\cdot\dim\Soc S.$
\end{theo}
\begin{proof}
Consider a subgroup $G=\GL(\USoc S)\subseteq\Aut S$. Along with the subgroup $U$ from Proposition \ref{Udim} it generates a subgroup $G U\subseteq\Aut S$. To prove the inequality
\begin{equation}
  \dim G U\ge\dim G+\dim U
\end{equation} it suffices to look at the tangent algebras of $G$ and $U$. Indeed, easy to see that they have a zero intersection, and% $\Lie(G U)=\Lie G\oplus\Lie U$ as a vector space. Hence
\begin{multline}
  \dim GU=\dim(\Lie GU)\ge\dim(\Lie G)+\dim(\Lie U)=\\
  \dim G+\dim U=(\dim(\USoc S))^2+\dim(\LSoc S)\cdot\dim(\m/\m^2) + \\
  \dim(\USoc S)\cdot(\dim(\m/\m^2)-\dim(\USoc S))=\dim(\Soc S)\cdot\dim(\m/\m^2).
\end{multline}
\end{proof}
\begin{corollary}
  The group $\Aut S$ is infinite if $S\neq\K$.
\end{corollary}

Clearly, the automorphism group almost always contains a rather big unipotent subgroup. A natural question arises if the whole automorphism group may be unipotent. The following proposition provides an example.
\begin{proposition}
    Consider the following local algebra
    \begin{equation}
      S=\K[x,y]/I,\quad I=(y^5,(x+y)^6,x^5-x^3y^3, x^4y).
    \end{equation}
Then the group $\Aut S$ is unipotent.
\end{proposition}
\begin{proof}
  The Gr\"obner basis of the ideal $I$ with respect to the homogeneous lexicographic order with $x\prec y$ is
  \begin{equation}
    \{x^6,y^5,x^3y^3-x^5,3x^2y^4+4x^5,x^4y\}.
  \end{equation}

  Clearly, $\mm^7\subset I$.
  Let $\x,\y$ be the images of $x,y$ respectively under factorization by $I$. The basis of algebra $S$ is as follows.

  \begin{tabular}{cccccc}\label{table}
    $1$&$\x$&$\x^2$&$\x^3$&$\x^4$&$(\x^5)$\\
    \vspace{3pt}
    $\y$&$\x\y$&$\x^2\y$&$\x^3\y$&&\\
    \vspace{3pt}
    $\y^2$&$\x\y^2$&$\x^2\y^2$&$\x^3\y^2$&&\\
    \vspace{3pt}
    $\y^3$&$\x\y^3$&$\x^2\y^3$&$(\x^3\y^3)$&&\\
    \vspace{3pt}
    $\y^4$&$\x\y^4$&$(\x^2\y^4)$&&&\\
  \end{tabular}

  where $-\frac{3}{4}\x^2\y^4=\x^3\y^3=\x^5$.

  Consider an arbitrary automorphism $\varphi\in\Aut S$,
  \begin{align}
    \varphi(\x)&=a_{11}\x+a_{12}\y+h_1(\x,\y),\quad h_1\in \mm^2,\\
    \varphi(\y)&=a_{21}\x+a_{22}\y+h_2(\x,\y),\quad h_2\in \mm^2.
  \end{align}
  Note that $\y$ is the only linear polynomial whose $5$th power degree is zero. Then $\varphi(\y^5)=0$ implies $a_{21}=0$.
  On the other hand, $\x$ is the only linear polynomial whose $5$th degree is not zero but lies in $\m^6$.
  Therefore,
  \begin{align}
    \varphi(\x)&=a_{11}\x+h_1(\x,\y),\quad h_1\in \mm^2,\\
    \varphi(\y)&=a_{22}\y+h_2(\x,\y),\quad h_2\in \mm^2,
  \end{align}
  where $a_{11},a_{22}\neq0$ due to invertibility of $\varphi$.
  Then we should notice that $\varphi((\x+\y)^6)=0$ if and only if $a_{11}=a_{22}=c$.
  Finally, $\varphi(x^3y^3-x^5)=c^6x^3y^3-c^5x^5=0$. It implies that $c=1$.

  Hence for an arbitrary element $z\in \m^i$ an inclusion $(\Aut S)\cdot z\subseteq z+\m^{i+1}$ holds. Thus, the group $\Aut S$ is unipotent.
\end{proof}
%
%\begin{definition}
%  An \emph{associated graded ideal} of an ideal $I\lhd\K[\![x_1,\ldots,x_n]\!]$ is a linear span $\mathcal{G}(I)$ of homogeneous components of the lowest degree of the polynomials of the ideal $I$. % REDO
%   Apparently, the set $\mathcal{G}(I)$ is a homogeneous ideal and it does not depend on the polynomial change of coordinates in $R$.
%\end{definition}
%\begin{remark}
% Note that $\mathcal{G}(I)$ may not be generated by homogeneous components of the lowest degree of the generators of the ideal $I$.
%\end{remark}
%\begin{proposition}
%$\gr S\cong R/\mathcal{G}(I)$.
%\end{proposition}
%\begin{proof}
%  Consider the $i$th homogeneous component
%  \begin{multline}
%    (\gr S)_i=\left.\m^i\right/\m^{i+1}=\left.\mm^i\right/(I\cap\mm^i+\mm^{i+1})=\left.\mm^i\right/((\mathcal{G}(I))_i+\mm^{i+1})=(R/\mathcal{G}(I))_i.
%  \end{multline}
% Together with the consistence of multiplication this equation implies the desired statement.
%\end{proof}

\section{Acknowledgements}
The author would like to express his gratitude to his scientific advisor I.~Arzhantsev for posing the problem, useful discussions and remarks. Thanks are also due to A. Elashvili for a thoughtful discussion and to M.~Zaidenberg for constructive remarks.
%Зайденбергу за стилистическую правку

\end{document}